\documentclass[submission,english]{dmlong}
\usepackage[T1]{fontenc}
\usepackage[latin9]{inputenc}
\usepackage{amssymb}
\usepackage{amsmath}

\makeatletter
  \newtheorem{thm}{\protect\theoremname}
  \newtheorem{lem}[thm]{\protect\lemmaname}

  \newtheorem{prop}[thm]{\protect\propositionname}
  
\newcommand{\ZZ}{{\mathbb Z}}
\newcommand{\FF}{{\mathbb F}}
\newcommand{\res}{{\rm res}}

\makeatother

\usepackage{babel}
  \providecommand{\corollaryname}{Corollary}
  \providecommand{\factname}{Remark}
  \providecommand{\lemmaname}{Lemma}
  \providecommand{\propositionname}{Proposition}
  \providecommand{\theoremname}{Theorem}

\begin{document}

\title[Counting roots of polynomials and roots of resultants]{Some results on counting roots of polynomials and the Sylvester resultant. }
\author{ Michael Monagan and Baris Tuncer }
\address{Department of Mathematics, Simon Fraser University, Burnaby, B.C., V5A 1S6, CANADA.}
\maketitle

\begin{abstract}  
\paragraph{Abstract.} \ We present two results, the first on the distribution
of the roots of a polynomial over the ring of integers modulo $n$ and the second
on the distribution of the roots of the Sylvester resultant of two multivariate
polynomials.
The second result has application to polynomial GCD computation and solving 
polynomial diophantine equations.

%\paragraph{R\'esum\'e.} \ Nous pr\'esentons deux r\'esultats, la premi\`ere sur la distribution
%des racines d'un polyn\^ome sur l'anneau des entiers modulo $n$ et la seconde
%sur la distribution des racines de la r\'esultante de deux variables multiples Sylvester
%polyn\^omes.  Le deuxi\`eme r\'esultat a une application pour polyn\^ome calcul
%de GCD et la r\'esolution \'equations diophantiennes polynomiales .

%\paragraph{R\'esum\'e. } \ Nous pr\'esentons deux r\'esultats: le premier concerne
%la distribution des racines d'un polyn\^ome sur l'anneau des entiers modulo $n$
%et le deuxi\`eme concerne la distribution des racines du d\'eterminant de Sylvester
%de deux polyn\^omes multivari\'es. Ceci est utile pour le calcul de PGCD et
%la r\'esolution des \'equations diophantiennes polynomiales.

%\paragraph{ Résumé. } \ Nous présentons deux résultats : le premier concerne 
% la distribution des racines d'un polynôme sur l'anneau des entiers modulo $n$ 
% et le deuxième concerne la distribution des racines du déterminant de Sylvester
% de deux polynômes multivariés. Ceci est utile pour le calcul de PGCD et
% la résolution des équations diophantiennes polynomiales.

\end{abstract}

\keywords{roots of polynomials, finite fields, the Sylvester resultant, unlucky evaluation points}

\section{Introduction}

\newcommand{\Var}{{\rm Var}}
\newcommand{\Prob}{{\rm Prob}}
\newcommand{\E}{{\rm E}}

Let $\FF_q$ denote the finite field with $q$ elements and 
let $\ZZ_n$ denote the ring of integers modulo $n$.
Let ${\rm E}[X]$ denote the expected value of a random variable $X$
and let $\Var[X]$ denote the variance of $X$.

Let $f$ be a polynomial in $\FF_q[x]$ of a given degree $d > 0$ and let $X$
be the number of distinct roots of $f$.
Schmidt proves in Ch. 4 of \cite{schmidt} that $\E[X]=1$ and for $d > 1$, $\Var[X] = 1-1/q$.
This result has been generalized by Knopfmacher and Knopfmacher in \cite{knopfmacher2}
who count distinct irreducible factors of a given degree of $f$.
The two main results presented in this paper are Theorems 1 and 2 below.

%%%%%%%%%%%%%%%%%%%%%%%%%%%%%%%%%%%%%%%%%%%%%%%%%%%%  THEOREM 1
\setcounter{thm}{0}
\begin{thm}
Let $\phi(n) = |\{ \, 1 \! \le \! i \! \le \! n : \gcd(i,n)=1 \}|$ denote Euler's totient function.
Let $X$ be a random variable which counts the number of distinct roots of a monic
polynomial in $\ZZ_n[x]$ of degree $m>0.$  Then \vspace*{-4mm}
\begin{itemize}
\item[(a)] ${\E}[X]=1$ \ and \vspace*{-2mm}
\item[(b)] if $m=1$ then $\mathrm{Var}[X]=0,$ otherwise 
$\mathrm{Var}[X]=\sum_{d|n,d\neq n}\frac{d}{n}\phi(\frac{n}{d})=\sum_{d|n}\frac{d-1}{n}\phi(\frac{n}{d})$. \\
In particular, if $n=p^{k}$
where $p$ is a prime number and $k\geq1$, $\mathrm{Var}[X]=k(1-1/p).$
\end{itemize}
\end{thm}

%%%%%%%%%%%%%%%%%%%%%%%%%%%%%%%%%%%%%%%%%%%%%%%%%%%%  THEOREM 2
\begin{thm}
Let $f,g$ be polynomials in $\FF_q[x,y]$ of the form
$f = c_n x^n + \sum_{i=0}^{n-1} \sum_{j=0}^{n-i} c_{ij} x^i y^j$ and
$g = d_m x^m + \sum_{i=0}^{m-1} \sum_{j=0}^{m-i} d_{ij} x^i y^j$ with
$c_n \ne 0$ and $d_m \ne 0$, thus of total degree $n$ and $m$ respectively.
Let $X$ be a random variable that counts the number of $\gamma\in\FF_{q}$
such that $\mathrm{gcd}(f(x,\gamma),g(x,\gamma))\neq 1$.
If $n>0$ and $m>0$ then \vspace*{-2mm}
\begin{itemize}
\item[(a)] $\mathrm{E}[X]=1$ \ and \vspace*{-2mm}
\item[(b)] $\mathrm{Var}[X]=1-1/q.$
\end{itemize}
\end{thm} %%\nopagebreak

Theorems 1 and 2 were found by computation.
We give some details on our computations later in the paper.
To prove the results we use a generalization of the 
Inclusion Exclusion principle (Proposition 1) which allows us to
determine ${\E}[X]$ and $\Var[X]$ without having explicit
formulas for $\Prob[X=k]$.  %\pagebreak
Before proving these results we connect Theorem 2 with the
Sylvester resultant and with polynomial GCD computation and with
solving polynomial diophantine equations.

Let $F$ be a field and let $A$ and $B$ be polynomials in $F[x_0,x_1,\dots,x_n]$ with positive degree in $x_0$.
The Sylvester resultant of $A$ and $B$ in $x_0$, denoted $\res_{x_0}(A,B)$, is the determinant of Sylvester's matrix.
We gather the following facts about it into Lemma 1 below.  
Proofs may be found in Ch. 3 of \cite{CLO}.
Note, in the Lemma $\deg A$ denotes the total degree of $A$.

\setcounter{thm}{0}

\begin{lem} Let $R = \res_{x_0}(A,B)$ 
\begin{itemize} \item[]

(i) $R$ is a polynomial in $F[x_1,\dots,x_n]$ \ ($x_0$ is eliminated), \\
(ii) $\deg R \le \deg A \deg B$ \ (Bezout bound).

For $A$ and $B$ monic in $x_0$ and $\alpha \in F^n$

(iii) $\gcd(A(x_0,\alpha),B(x_0,\alpha)) \ne 1 \iff \res_{x_0}(A(x_0,\alpha),B(x_0,\alpha)) = 0$ and \\
(iv) $\res_{x_0}(A(x_0,\alpha),B(x_0,\alpha)) = R(\alpha)$.
\end{itemize}
\end{lem}

\noindent
Properties (iii) and (iv) connect the roots of the resultant with Theorem 2 and 3.

\subsection{Polynomial GCD computation and polynomial diophantine equations.}

Our motivation comes from the following problems in computer algebra.
Let $A, B$ be polynomials in $\ZZ[x_0,x_1,\dots,x_n]$ and $G = \gcd(A,B)$.
Thus $A=G\widehat A$ and $B=G\widehat B$ for some polynomials $\widehat A$ and $\widehat B$ called
the cofactors of $A$ and $B$.
Modular GCD algorithms compute $G$ modulo a sequence of primes $p_1,p_2,p_3,\dots$
and recover the integer coefficients of $G$ using Chinese remaindering.  
The fastest algorithms for computing $G$ modulo a prime $p$ interpolate $G$ from univariate images.
Maple, Magma and Mathematica all currently use Zippel's algorithm (see \cite{zippel79,wittkopf05}).
Let us write
$$A = \sum_{i=0}^k a_i x_0^i, 
~~B = \sum_{i=0}^l b_i x_0^i, ~~{\rm and}
~~G = \sum_{i=0}^m c_i x_0^i$$ 
where the coefficients $a_i,b_i,c_i \in \FF_p[x_1,\dots,x_n]$.
Zippel's algorithm picks points $\alpha_i \in \FF_p^n$, computes monic univariate images of $G$
$$ g_i = \gcd(A(x_0,\alpha_i),B(x_0,\alpha_i)),$$
scales them (details omitted), then interpolates the coefficients $c_i(x_1,\dots,x_n)$ of $G$ from
the coefficients of these (scaled) images.

\medskip
What if $\gcd( \widehat A(x_0,\alpha_j),\widehat B(x_0,\alpha_j)) \ne 1$ for some $j$?
For example, if $\widehat A = x_0^2 + x_2$ and $\widehat B = x_0^2 + x_2 + (x_1-1)$ then
$\gcd(\widehat A,\widehat B)=1$ but $\gcd(\widehat A(x_0,1,\beta),\widehat B(x_0,1,\beta)) \ne 1$ for
all $\beta \in \FF_p$.  The evaluation points $(1,\beta)$ are said to be {\em unlucky}.
We cannot use the images $\gcd(A(x_0,1,\beta),B(x_0,1,\beta))$ to interpolate $G$.
The same issue of unlucky evaluation points arises in our current work in \cite{tuncer}
where, given polynomials $a,b,c \in \ZZ[x_0,x_1,\dots,x_n]$ with $\gcd(a,b)=1$
we want to solve the diophantine equation $\sigma a + \tau b = c$
for $\sigma$ and $\tau$ in $\ZZ[x_0,x_1,\dots,x_n]$ by interpolating $\sigma$ and $\tau$
modulo a prime $p$ from univariate images.

\medskip
What is the maximum number of unlucky evaluation points that can occur?
And what is the expected number of unlucky evaluation points?
We answer the first question for $A$ and $B$ monic in $x_0$.  
Lemma 1 implies $\alpha_j$ is unlucky if and only if $R(\alpha_j)=0$ where 
$R = \res_{x_0}(\widehat A,\widehat B) \in \FF_p[x_1,\dots,x_n].$
If $\alpha_j$ is chosen at random from $\FF_p^n$ then applying the Schwarz-Zippel
lemma (see \cite{SZlemma}) we have
$$\Prob[ \, R(\alpha_j)=0 \, ] \le \frac{\deg R}{p}.$$
Applying Lemma 1(ii) we have $\deg R \le \deg \widehat A \deg \widehat B \le \deg A \deg B$.
So if the algorithm needs, say, $t$ images to interpolate $G$ modulo $p$,
then we can avoid unlucky evaluation points with high probability 
if we pick $p \gg t \deg A \deg B$.

But this is an upper bound -- a worst case bound for the GCD algorithm.
Researchers in computer algebra have observed that unlucky evaluation points are
rare in practice and that we ``never see them'' when testing algorithms on random inputs.
Theorems 2 and 3 give first results on the distribution of unlucky evaluation points.
In particular, for coprime $\widehat A$ and $\widehat B$ of positive degree,
Theorem 3 (page 11) implies $\Prob[ \, \alpha_j {\rm ~is~unlucky} \, ] < 1/p $.

\section{Results and Proofs}
Given a set $U$ and the finite collection of sets $\Gamma=\{A_{i},i=0,\ldots,n-1\}$
where each $A_{i}\subseteq U$, let us define $C_0 = U$, $C_{n+1}:=\emptyset$
and, for $1\leq k\leq n$,
\[
C_{k}:=\bigcup_{i_{1}<\cdots<i_{k}}(A_{i_{1}}\cap A_{i_{2}}\cdots\cap A_{i_{k}}).
\]
Then for $1\leq k\leq n$, $C_{k}$ is the union of all possible
intersections of the $k-$subsets of the collection $\Gamma$ .
In particular $C_1 = A_0 \cup A_1 \cup \cdots \cup A_{n-1}$ and  $C_n = A_0 \cap A_1 \cap \cdots \cap A_{n-1}$.
Let $B_{k}:=C_{k}-C_{k+1}$ for $0 \le k \le n$.
Observe that $C_{k}\supseteq C_{k+1}$, so $|B_{k}|=|C_{k}|-|C_{k+1}|$. Let
us also define
\[
b_{k}:=|B_{k}|\,\,\mathrm{and}\,\,t_{k}:=\sum_{i_{1}<\cdots<i_{k}}|A_{i_{1}}\cap A_{i_{2}}\cdots\cap A_{i_{k}}|.
\]
We have $t_1 = \sum_{i=0}^{n-1} |A_i|$ and $t_2 = \sum_{0\le i<j<n} |A_i \cap A_j|.$
We also have $b_{n}=t_{n}$ and $b_{n-1}=t_{n-1}-\binom{n}{1}b_{n}$.  Now
$A_{0}\cap A_{1}\cap\cdots\cap A_{n-1}$ is a subset of $\binom{n}{n-2}=\binom{n}{2}$
sets of the form $A_{i_{1}}\cap A_{i_{2}}\cap\cdots\cap A_{i_{n-2}}$
and each $(n-1)$-section $A_{i_{1}}\cap A_{i_{2}}\cap\cdots\cap A_{i_{n-1}}$
is a subset of $\binom{n-1}{n-2}=\binom{n-1}{1}$ sets of the form
$A_{i_{1}}\cap A_{i_{2}}\cap\cdots\cap A_{i_{n-2}}$ with $i_{1}<i_{2}<\cdots<i_{n-2}$.
Therefore $b_{n-2}=t_{n-2}-\binom{n-1}{1}b_{n-1}-\binom{n}{2}b_{n}$. 

Similarly, since each $(n-k+i)$-section is a subset of $\binom{n-k+i}{i}$
intersections of $(n-k)$ sets for $i=1,\ldots,k$, we have the recursive
formula
\begin{equation}
b_{n-k}=t_{n-k}-\sum_{i=1}^{k}{\binom{n-k+i}{i}}b_{n-k+i}\,\,\mathrm{for}\,\,k=0,\ldots,n.
\label{eq:recursive}
\end{equation}

\begin{lem}
Following the notation introduced above
\begin{equation}
b_{n-k}=\sum_{i=0}^{k}(-1)^{i}{\binom{n-k+i}{i}}t_{n-k+i}\,\,\mathrm{for}\,\,k=0,\ldots,n.
\label{eq:basic}
\end{equation}
\end{lem}
\begin{proof}
We will prove the claim by strong induction on $k$. For $k=0$ we have $b_{n}=t_{n}.$
Now assume that the claim is true for any integer $i \leq k$ in place of $k$.

By the recursive formula (\ref{eq:recursive}) we have

% \begin{multline*}
% (\mathrm{by\,(1))}\,\,b_{n-(k+1)}=t_{n-(k+1)}+\sum_{i=1}^{k+1}\binom{n-(k+1)+i}{i}b_{n-(k+1)+i}\\
% (\mathrm{subs.}\,\,i\mapsto i+1)=t_{n-(k+1)}+\sum_{i=0}^{k}\binom{n-k+i}{i+1}b_{n-k+i}\\
% (\mathrm{by\,ind.\,hyp.})=t_{n-(k+1)}-\sum_{i=0}^{k}\binom{n-k+i}{i+1}\sum_{j=0}^{k-i}(-1)^{j}\binom{n-(k-i)+j}{j}t_{n-(k-i)+j}\\
% (\mathrm{subs.}\,\,j\mapsto j-i)=t_{n-(k+1)}-\sum_{i=0}^{k}\sum_{j=i}^{k}(-1)^{j-i}\binom{n-k+i}{i+1}\binom{n-k+j}{j-i}t_{n-k+j}\\
% =t_{n-(k+1)}-\sum_{j=0}^{k}\biggl(\sum_{i=0}^{j}(-1)^{j-i}\binom{n-k+i}{i+1}\binom{n-k+j}{j-i}\biggr)t_{n-k+j}\\
% (\mathrm{subs.}\,\,j\mapsto j-1)=t_{n-(k+1)}-\sum_{j=1}^{k+1}\biggl(\sum_{i=0}^{j-1}(-1)^{j-1-i}\binom{n-k+i}{i+1}\binom{n-k+j-1}{j-1-i}\biggr)t_{n-(k+1)+j}\\
% =t_{n-(k+1)}-\sum_{j=1}^{k+1}(-1)^{j+1}\binom{n-(k+1)+j}{j}t_{n-(k+1)+j}\\
% =\sum_{j=0}^{k+1}(-1)^{j}\binom{n-(k+1)+j}{j}t_{n-(k+1)+j}.
% \end{multline*}

\[
\textstyle
b_{n-(k+1)}=t_{n-(k+1)}-\binom{n-k}{1}b_{n-k}-\binom{n-k+1}{2}b_{n-k+1}-\cdots-\binom{n}{k+1}b_{n}.
\]

%\newpage
\noindent
On the other hand by induction we have the following equations

\bigskip

$b_{n}\:\:\:\ =t_{n}$

$b_{n-1}=t_{n-1}-\binom{n}{1}t_{n}$

$b_{n-2}=t_{n-2}-\binom{n-1}{1}t_{n-1}+\binom{n}{2}t_{n}$

$\vdots$

$b_{n-k}=t_{n-k}-\binom{n-k+1}{1}t_{n-k+1}+\cdots+(-1)^{k}\binom{n}{k}t_{n}.$
\\
\\
It follows that\\

$-\binom{n}{k+1}b_{n}\,\,\,\,\,\,=-\binom{n}{k+1}t_{n}$

$-\binom{n-1}{k}b_{n-1}=-\binom{n-1}{k}t_{n-1}+\binom{n-1}{k}\binom{n}{1}t_{n}$

$-\binom{n-2}{k-1}b_{n-2}=-\binom{n-2}{k-1}t_{n-2}+\binom{n-2}{k-1}\binom{n-1}{1}t_{n-1}-\binom{n-2}{k-1}\binom{n}{2}t_{n}$

$\vdots$

$-\binom{n-k}{1}b_{n-k}=-\binom{n-k}{1}t_{n-k}+\binom{n-k}{1}\binom{n-k+1}{1}t_{n-k+1}-\cdots(-1)^{k+1}\binom{n-k}{1}\binom{n}{k}t_{n}.$ 
\\
\\
If we sum all these equalities, then on the right hand side the coefficient of $t_{n}$ is \\
\\
 $c(t_{n})=\sum_{i=0}^{k}(-1)^{k-i+1}\binom{n-k+i}{i+1}\binom{n}{k-i}$.
For $d\leq k$ one has\\
\\
 $\binom{n-d}{k-d+1}\binom{n}{d}=\frac{(n-d)!}{(n-k-1)!(k-d+1)!}\frac{n!}{(n-d)!d!}=\frac{n!}{(k+1)!(n-k-1)!}\frac{(k+1)!}{d!(k-d+1)!}=\binom{n}{k+1}\binom{k+1}{d}$.
\\
\\
Then $c(t_{n})=\binom{n}{k+1}\sum_{i=0}^{k}(-1)^{k-i+1}\binom{k+1}{k-i}$$=-\binom{n}{k+1}(-1)^{k}=(-1)^{k+1}\binom{n}{k+1},$
\\
\\
where the last equality follows from the fact that \\
\\
$\binom{k+1}{0}-\binom{k+1}{1}+\binom{k+1}{2}+\cdots+(-1)^{k}\binom{k+1}{k}=-(-1)^{k+1}=(-1)^{k}$
. \\
\\
Similarly for $s=1,\ldots,k$ we have $c(t_{n-s})=\sum_{i=0}^{k-s}(-1)^{k-s-i+1}\binom{n-s-k+i}{i+1}\binom{n}{k-s-i}$
\\
\\
$=\binom{n-s}{k-s+1}\sum_{i=0}^{k-s}(-1)^{k-s-i+1}\binom{n}{k-s-i}=(-1)^{k-s+1}\binom{n-s}{k-s+1}.$\\
\\
Now plugging $s=k-i$  in the formula above we get\\ 
\\
$b_{n-(k+1)}=\sum_{i=0}^{k}(-1)^{i+1}\binom{n-k+i}{i+1}t_{n-k+i}$.
\end{proof}
% 
% \newpage

\setcounter{thm}{0}

\begin{prop}
Following the same notation one has for $1\leq\ensuremath{k\leq n},$\\
$\ensuremath{\sum_{i=0}^{n}i^{k}b_{i}=\sum_{i=1}^{k}i^{k}\left[\sum_{j=i}^{k}(-1)^{j-i}{j \choose j-i}t_{j}\right]},$
In particular:

\begin{itemize}
\item[(a)] $\sum_{i=0}^{n}ib_{i}=t_{1}=\sum_{i=0}^{n-1}|A_{i}|$ \ (Inclusion Exclusion Principle) and

\item[(b)] $\sum_{i=0}^{n}i^{2}b_{i}=t_{1}+2t_{2}=\sum_{i=0}^{n-1}|A_{i}|+2\sum_{i<j}|A_{i}\cap A_{j}|.$ 
\end{itemize}
\end{prop}

\begin{proof}
%Since $\sum_{i=0}^{n}ib_{i}=\sum_{i=1}^{n}ib_{i}$ and $\sum_{i=0}^{n}i^{2}b_{i}=\sum_{i=1}^{n}i^{2}b_{i}$
%we will prove the statements assuming that the index $i$ starting
%from 1.
%
According to Lemma 2 we have \\

$b_{n}=t_{n}$

$b_{n-1}=t_{n-1}-\binom{n}{1}t_{n}$

$b_{n-2}=t_{n-2}-\binom{n-1}{1}t_{n-1}+\binom{n}{2}t_{n}$

$\vdots$

$b_{2}=t_{2}-\binom{3}{1}t_{3}+\binom{4}{2}t_{4}+\cdots
    +(-1)^{n-3} \binom{n-1}{n-3} t_{n-1} + (-1)^{n-2}\binom{n}{n-2}t_{n}$

$b_{1}=t_{1}-\binom{2}{1}t_{2}+\binom{3}{2}t_{3}+\binom{4}{3}t_{4}+\cdots
  + (-1)^{n-2}\binom{n-1}{n-2} t_{n-1} + (-1)^{n-1}\binom{n}{n-1}t_{n}$\\
\\
If we sum $\sum_{i=1}^{n}i^{k}b_{i}$, then for $1\leq s\leq n,$
the coefficient of $t_{s}$ on the right hand side is \\

$c(t_{s})=\sum_{i=1}^{s}i^{k}\binom{s}{s-i}(-1)^{s-i}.$  \\
\\
We claim that $c(t_s)=0$ for $k \! < \! s \! \leq  \! n$.
We prove this by strong induction on $k$.  For $k \! = \! 1$ we have \\

$c(t_{s})=\sum_{i=1}^{s}i\binom{s}{s-i}(-1)^{s-i}=\sum_{i=1}^{s}i\binom{s}{i}(-1)^{s-i}=\sum_{i=1}^{s}s\binom{s-1}{i-1}(-1)^{s-i}$
\\
\\
Since $s\geq2,$ by substituting $m=s-1\geq1$ and $j=i-1$ \\

$c(t_{s})=s\sum_{j=0}^{m}\binom{m}{j}(-1)^{m-j}=s(1-1)^{m}=0$.\\
\\
Now assume $\sum_{i=1}^{s}i^{l}\binom{s}{s-i}(-1)^{s-i}=0$
for any $1\leq l\leq k$ and $l+1\leq s\leq n.$ Then \\

$c(t_{s})=\sum_{i=1}^{s}i^{k+1}\binom{s}{s-i}(-1)^{s-i}=\sum_{i=1}^{s}i^{k+1}\binom{s}{i}(-1)^{s-i}=\sum_{i=1}^{s}si^{k}\binom{s-1}{i-1}(-1)^{s-i}.$\\
\\
Substituting $m=s-1\geq l\geq1$ and $j=i-1$ we obtain \\

$c(t_{s})=s\sum_{j=0}^{m}(j+1)^{k}\binom{m}{j}(-1)^{m-j}=s\sum_{j=0}^{m}\sum_{l=0}^{k}\binom{k}{l}j^{l}\binom{m}{j}(-1)^{m-j}$\\

\hspace*{7mm} $=s\sum_{l=0}^{k}\binom{k}{l}\sum_{j=0}^{m}j^{l}\binom{m}{j}(-1)^{m-j}.$\\
\\
Since $m=s-1\geq l$, we have $m\geq l+1$ and by induction hypothesis
each summand\\

$\sum_{j=0}^{m}j^{l}\binom{m}{j}(-1)^{m-j}=\sum_{j=1}^{m}j^{l}\binom{m}{m-j}(-1)^{m-j}=0.$ \\
\\
Hence $c(t_{s})=0.$
On the other hand for $1\leq s\leq k$ the coefficient of each $i^{s}$
on the right hand side is $\sum_{j=i}^{k}(-1)^{j-i}{j \choose j-i}t_{j}$.
Hence we have the result.
In particular, for $k=2$ the non-zero terms on the right-hand-side are
$t_{1}-\binom{2}{1} t_2 + 2^{2}t_{2}=t_{1}+2t_{2}$.
 \end{proof}

\setcounter{thm}{0}
\newpage
%%%%%%%%%%%%%%%%%%%%%%%%%%%%%%%%%%%%%%%%%%%%%%%%%%%%  COROLLARY 1
\begin{thm}
Let $\phi(n) = |\{ \, 1 \! \le \! i \! \le \! n : \gcd(i,n)=1 \}|$ denote Euler's totient function.
Let $X$ be a random variable which counts the number of distinct roots of a monic
polynomial in $\ZZ_n[x]$ of degree $m>0.$  Then \vspace*{-1mm}
\begin{itemize}
\item[(a)] ${\rm E}[X]=1$ \ and \\ \vspace*{-2mm}
%\item[(a)] ${\E}[X]=1$ \ and \vspace*{-2mm}
\item[(b)] if $m=1$ then $\mathrm{Var}[X]=0,$ otherwise 
$\mathrm{Var}[X]=\sum_{d|n,d\neq n}\frac{d}{n}\phi(\frac{n}{d})=\sum_{d|n}\frac{d-1}{n}\phi(\frac{n}{d})$. \\
In particular, if $n=p^{k}$
where $p$ is a prime number and $k\geq1$, $\mathrm{Var}[X]=k(1-1/p).$
\end{itemize}
\end{thm}

\noindent
{\bf Remark 1.} \ 
We found this result by direct computation and using the
Online Encylopedia of Integer Sequences (OEIS) see \cite{OEIS}.
For polynomials of degree 2,3,4,5 in $\ZZ_n[x]$ we computed ${\rm E}[X]$
and ${\rm Var}[X]$ for $n=2,3,4,\dots,20$ using Maple and found 
that ${\rm E}[X]=1$ in all cases.
Values for the variance are given in the table below.
\medskip

\begin{tabular}{|l|ccccccccccccccc|} \hline
$n$              &  2 &  3  & 4 &  5  & 6  & 7  & 8  &  9  &  10  &  11   &  12  &  13  &  14 &  15  &  16 \\
%$Var[X]$        & 1/2& 2/3 & 1 & 4/5 &3/2 &6/7 &3/2 & 4/3 &17/10 & 10/11 & 7/3 &  12/13 &  25/14 &  2 &  2  \\
Var[$X$]         & $\frac{1}{2}$ & $\frac{2}{3}$ & 1 & $\frac{4}{5}$&$\frac{3}{2}$&$\frac{6}{7}$&$\frac{3}{2}$& $\frac{4}{3}$&$\frac{17}{10}$& $\frac{10}{11}$& $\frac{7}{3}$ & $\frac{12}{13}$ & $\frac{25}{14}$ & 2 & 2 \\
$a(n)$           &  1 &  2  & 4 &  4  & 9  & 6  & 12 &  12 &  17  &  10   &  28  &  12  &  25 &  30  &  32 \\ \hline
\end{tabular}

\medskip
\noindent
When we first computed ${\rm Var}[X]$ we did not recognize the numbers.
Writing ${\rm Var}[X] = a(n)/n$ we computed the sequence for $a(n)$ (see the table)
and looked it up in the OEIS.  We found it is sequence A006579 and
that $a(n) = \sum_{k=1}^{n-1} \gcd(n,k)$.  The OEIS also has the formula
$a(n) = \sum_{d|n} (d-1) \phi(\frac{n}{d}).$

\begin{proof} \quad 
Let $A_{i}$ be the set of all monic univariate polynomials of degree
$m>0$ which have a root at $\alpha_{i}\in\mathbb{Z}_{n}$. Then since
$x-\alpha_i$ is monic, for any $f\in A_{i}$ we have $f=(x-\alpha_{i})q$
for a unique $q\in \ZZ_{n}[x]$ and we have $n^{m-1}$ choices for such
an $f$. Hence $|A_{i}|=n^{m-1}.$ 

%Let X be a random variable which counts the number of roots of a monic
%polynomial $f$ of degree $m>0$ without multiplicity in $\mathbb{Z}_{n}.$
Let $x_{i}:=\mathrm{Prob}[X=i]$. This is the probability that
$f$ has exactly $i$ distinct roots, i.e. $f\in B_{i}$ in the notation
introduced in section 1 considering the finite collection of sets
$\Gamma=\{A_{i},i=0,\ldots,n-1\}$. Since we have $n^{m-1}$ choices
for a monic polynomial of degree $m$ in $\ZZ_{n}[x]$ we have $x_{i}=\frac{b_{i}}{n^{m-1}}$.
Then by Proposition 1
\[
{\rm E}[X]=\sum_{i=0}^{n}ix_{i}=\sum_{i=0}^{n}i\frac{b_{i}}{n^{m}}=\frac{\sum_{i=0}^{n}ib_{i}}{n^{m}}=\frac{\sum_{i=0}^{n-1}|A_{i}|}{n^{m}}=\frac{\sum_{i=0}^{n-1}n^{m-1}}{n^{m}}=\frac{nn^{m-1}}{n^{m}}=1
\]

%Then just as in the proof of Corollary 3 (a) $\mathrm{E}[X]=1$.
\noindent
To prove (b), if $m=1$ then $f = x-\alpha$ for
some $\alpha\in \ZZ_n$ and hence $X=1$ and $\Var[X]=0$.
For $m>1$ and $\alpha\in\mathbb{Z}_{n}^{*}$, our first aim
is to find $|A_{0}\cap A_{\alpha}|$. Let $f\in A_{0}\cap A_{\alpha}$.
It may not be the case that $f=x(x-\alpha)q$ for a unique $q\in\mathbb{Z}_{n}[x]$,
since $\mathbb{Z}_{n}[x]$ is not a unique factorization domain in
general. However $f=xq_{1}=(x-\alpha)q_{2}$ for unique $q_{1},q_{2}\in\mathbb{Z}_{n}[x]$.
It follows that $\alpha q_{2}(0)=0\,\,\mathrm{mod}\,\,n.$
If $\gcd(\alpha,n)=d$ then $\gcd(\frac{\alpha}{d},\frac{n}{d})=1$
and hence $q_{2}(0)=0\,\,\mathrm{mod}\,\,\frac{n}{d}$. The general form
of $q_{2}=x^{m-1}+a_{m-2}x^{m-2}+\cdots+a_{0}$ where $a_{i}\in\mathbb{Z}_{n}$
for $i=0,\ldots,m-2$. Since $q_{2}(0)=a_{0}\:\;\mathrm{mod}\:\;\frac{n}{d},$
there are $d$ choices for $a_{0}$ and hence there are $dn^{m-2}$
choices for $q_{2}.$ Therefore $|A_{0}\cap A_{\alpha}|=dn^{m-2}.$

For a given pair $(\gamma,\beta)$ with $\beta>\gamma,$ to compute
$|A_{\gamma}\cap A_{\beta}|$, define $\alpha:=\beta-\gamma$ and
consider $A_{0}\cap A_{\alpha}.$ If $f\in A_{\gamma}\cap A_{\beta},$
then we have $f(x)=(x-\gamma)q_{3}(x)=(x-\beta)q_{4}(x)$ for unique
$q_{3},q_{4}\in \ZZ_{n}[x]$. By the coordinate translation $x\mapsto x+\gamma$
we have $f(x+\gamma)\in A_{0}\cap A_{\alpha}$, since $f(x+\gamma)=xq_{3}(x+\gamma)=(x-\alpha)q_{4}(x+\gamma)$
where $f(x+\gamma),q_{3}(x+\gamma),q_{4}(x+\gamma)$ are monic
and with the same degree before the translation. This correspondence
is bijective and it follows that $|A_{\gamma}\cap A_{\beta}|=|A_{0}\cap A_{\alpha}|=dn^{m-2}.$ 

Let $d=\gcd(\alpha,n).$ There are $k=\phi(\frac{n}{d})$
elements $\beta_{1},\ldots,\beta_{k}$ in $\mathbb{Z}_{\frac{n}{d}}$
such that $\mathrm{gcd}(\beta_{j},\frac{n}{d})=1.$
If we define $\alpha_{j}:=d\beta_{j}\in\mathbb{Z}_{n}$
then $\mathrm{gcd}(\alpha_{j},n)=d.$ For, if $s=\mathrm{gcd}(\alpha_{j},n)$
and $d|s$ then $s|\alpha_{j}\Rightarrow s|d\beta_{j}\Rightarrow\frac{s}{d}|\beta_{j}$
and $\frac{s}{d}|\frac{n}{d}\Rightarrow\frac{s}{d}|\mathrm{gcd}(\beta_{j},\frac{n}{d})\Rightarrow\frac{s}{d}|1\Rightarrow s=d.$
Now, for each $j$ consider the $n-\alpha_{j}$ pairs of the form
$(i,i+\alpha_{j})$ where $i=0,\ldots,n-\alpha_{j}-1$. We have $|A_{i}\cap A_{i+\alpha_{j}}|=|A_{0}\cap A_{\alpha_{j}}|$
and 
\[
\sum_{\beta>\gamma,d=\mathrm{gcd}(\beta-\gamma,n)}|A_{\gamma}\cap A_{\beta}|=\sum_{j=1}^{k}(n-\alpha_{j})|A_{0}\cap A_{\alpha_{j}}|=\sum_{j=1}^{k}(n-\alpha_{j})dn^{m-2}=dn^{m-2}\sum_{j=1}^{k}n-\alpha_{j}
\]
where $d=\mathrm{gcd}(\alpha_{j},n)$ and $k=\phi(\frac{n}{d}).$
Since $\mathrm{gcd}(n,\alpha_{j})=d \iff \mathrm{gcd}(n,n-\alpha_{j})=d$
we have $\sum_{j=1}^{k}n-\alpha_{j}=\sum_{j=1}^{k}\alpha_{j}$. Then
\[
2\sum_{j=1}^{k}\alpha_{j}=\sum_{j=1}^{k}\alpha_{j}+\sum_{j=1}^{k}n-\alpha_{j}=\sum_{j=1}^{k}n=kn=\phi(\frac{n}{d})n
~~ \Longrightarrow ~~
\sum_{j=1}^{k}\alpha_{j}=\frac{n}{2}\phi(\frac{n}{d}).
\]
It follows that
\[
\sum_{\beta>\gamma,d=\mathrm{gcd}(\beta-\gamma,n)}|A_{\gamma}\cap A_{\beta}|=dn^{m-2}\sum_{j=1}^{k}n-\alpha_{j}=dn^{m-2}\sum_{j=1}^{k}\alpha_{j}=\frac{n}{2}\phi(\frac{n}{d})dn^{m-2}.
\]
Then by Proposition 1 it follows that 
\begin{eqnarray*}
\mathrm{Var}[X] & = & \mathrm{E}[X^{2}]-\mathrm{E}[X]^{2}=-1^2+\mathrm{E}[X^{2}]\\
 & = & -1+\sum_{i=0}^{n}i^{2}x_{i}=-1+\sum_{i=0}^{n}i^{2}\frac{b_{i}}{n^{m}}=-1+\frac{\sum_{i=0}^{n}i^{2}b_{i}}{n^{m}}\\
 & = & -1+\frac{\sum_{i=0}^{n-1}|A_{i}|+2\sum_{i<j}|A_{i}\cap A_{j}|}{n^{m}}\\
% & = & -1+\frac{\sum_{i=0}^{n-1}|A_{i}|}{n^{m}}+\frac{2\sum_{d|n\,d\neq n}\frac{n}{2}\phi(\frac{n}{d})dn^{m-2}}{n^{m}}\\
 & = & -1+\frac{nn^{m-1}}{n^{m}}+\frac{2\sum_{d|n\,d\neq n}\frac{n}{2}\phi(\frac{n}{d})dn^{m-2}}{n^{m}}\\
 & = & 2\sum_{d|n\,d\neq n}\frac{n}{2}\phi(\frac{n}{d})dn^{-2}
 ~ = ~ \sum_{d|n\,d\neq n}\frac{d}{n}\phi(\frac{n}{d}).
\end{eqnarray*}
Also, since by Gauss' Lemma $\sum_{d|n}\phi(\frac{n}{d}) = n$ we have
\begin{eqnarray*}
\sum_{d|n}\frac{d-1}{n}\phi(\frac{n}{d}) & = & \sum_{d|n}\frac{d}{n}\phi(\frac{n}{d})-\frac{1}{n}\sum_{d|n}\phi(\frac{n}{d})\\
 & = & \phi(1)+\sum_{d|n,d\neq n}\frac{d}{n}\phi(\frac{n}{d})-\frac{1}{n} n 
 %& = & 1 + \sum_{d|n,d\neq n}\frac{d}{n}\phi(\frac{n}{d}) - \frac{1}{n} n
 ~ = ~ \sum_{d|n,d\neq n}\frac{d}{n}\phi(\frac{n}{d}).
\end{eqnarray*}
To prove the last claim, let $n=p^{k}$ where $p$ is a prime number
and $k\geq1$. Then
\begin{eqnarray*}
\sum_{d|n,d\neq n}\frac{d}{n}\phi(\frac{n}{d}) & = & \sum_{s=0}^{k-1}\frac{p^{s}}{p^{k}}\phi(\frac{p^{k}}{p^{s}})=\sum_{s=0}^{k-1}p^{s-k}p^{k-s-1}(p-1)=k(1-1/p).
\end{eqnarray*}
\end{proof}

%\newpage
%%%%%%%%%%%%%%%%%%%%%%%%%%%%%%%%%%%%%%%%%%%%%%%%%%%%  COROLLARY 5
\begin{thm}

Let $f,g$ be polynomials in $\FF_q[x,y]$ of the form
$f = c_n x^n + \sum_{i=0}^{n-1} \sum_{j=0}^{n-i} c_{ij} x^i y^j$ and
$g = d_m x^m + \sum_{i=0}^{m-1} \sum_{j=0}^{m-i} d_{ij} x^i y^j$ with
$c_n \ne 0$ and $d_m \ne 0$, thus of total degree $n$ and $m$ respectively.
Let $X$ be a random variable that counts the number of $\gamma\in\FF_{q}$
such that $\mathrm{gcd}(f(x,\gamma),g(x,\gamma))\neq 1$.
If $n>0$ and $m>0$ then

%For a given $f,g\in\FF_{q}[x,y]$ where $\FF_q$ is the finite field with $q$ elements
%and $f,g$ are monic in $x$ with total degrees $n>0$ and $m>0$ respectively, let $X$ be a random variable
%which counts the number of $\gamma\in\FF_{q}$ such that $\mathrm{gcd}(f(x,\gamma),g(x,\gamma))\neq1.$  
%Then

\begin{itemize}
\item[(a)] $\mathrm{E}[X] = 1$ and \vspace*{-1mm}
\item[(b)] $\mathrm{Var}[X] = 1-1/q.$ 
\end{itemize}
\end{thm}

% Need to allow f and g to be non-monic 
% Done

\noindent
{\bf Remark 2.} \ We found this result by computation.
For quadratic polynomials $f,g$ of the form
$f = x^2 + (a_1 y + a_2) x + a_3 y^2 + a_4 y + a_5$
and 
$g = x^2 + (b_1 y + b_2) x +  b_3 y^2 + b_4 y + b_5$
over finite fields of size $q=2,3,4,5,8,9,11$ we generated
all $q^{10}$ pairs and computed
$X = \left| \{ \alpha \in \FF_q : \gcd(f(x,\alpha),g(x,\alpha)) \ne 1 \} \right|.$
Magma code for $\FF_4$ is given in Appendix A.
We repeated this for cubic polynomials and some higher degree bivariate polynomials
for $q=2,3$ to verify that ${\rm E}[X]=1$ and ${\rm Var}[X]=1-1/q$
holds more generally.  
For yet higher degree polynomials we used random samples.
That ${\rm E}[X]=1$ independent of the degrees of $f$ and $g$ was a surprise to us.
We had expected a logarithmic dependence on the degrees of $f$ and $g$.

\medskip \noindent
{\bf Proof:} \quad
Without loss of generality we may assume $f$ and $g$ are monic in $x$ because
$\gcd(f(x,\gamma),g(x,\gamma)) = 1 \iff \gcd( c_n^{-1} f(x,\gamma), d_m^{-1} g(x,\gamma) ) = 1$.
For $\gamma\in\FF_{q},$ let us define $A_{\gamma}$ as the
set of polynomial pairs $(f,g)\in\mathbb{F}_{q}[x,y]^{2}$ where $f,g$
are monic in $x$ with total degrees, $\mathrm{deg}(f)=n>0$ and $\mathrm{deg}(g)=m>0$
such that $\mathrm{gcd}(f(x,\gamma),g(x,\gamma))\neq1.$ Our first
aim is to compute $|A_{0}|$.

Let $(f,g)\in A_{0}$. Since $f$ and $g$ are monic in $x$, $f(x,0),g(x,0)$
are monic polynomials of degree $n$ and $m$ respectively in $\mathbb{F}_{q}[x].$
We have finitely many choices, say $s,$ for non-relatively prime
monic polynomial pairs $(h_{i}(x),l_{i}(x))$ with $\mathrm{deg}(h_{i})=n$
and $\mathrm{deg}(l_{i})=m$ with $i=1,\ldots,s$$ $ in $\mathbb{F}_{q}[x]^{2}.$
Let $(f(x,0),g(x,0))=(h_{i}(x),l_{i}(x))$ for some fixed $i$ where
$1\leq i\leq s$. In fact $s=(q^{n}q^{m})/q=q^{n+m-1}$, since there
are $q^{n}q^{m}$ possible choices for monic polynomial pairs $(h,l)$
in $\mathbb{F}_{q}[x]$ with $\mathrm{deg}(h)=n$, $\mathrm{deg}(l)=m$
and the probability of a given monic pair is non-relatively prime
over $\mathbb{F}_{q}[x]$ is $1/q$ (see \cite{panario,berlekamp} and also
\cite{benjamin} for an accessible proof).  

Let $f(x,y)=x^{n}+c_{n-1}(y)x^{n-1}+\cdots+c_{1}(y)x+c_{0}(y)$
where $c_{d}(y)\in\mathbb{F}_{q}[y]$ of total degree $\mathrm{deg}(c_{n-d}(y))\leq d$
and let $c_{n-d}(y)=a_{d}^{(n-d)}y^{d}+\cdots+a_{0}^{(n-d)}$
where $a_{i}^{(n-d)}\in\mathbb{F}_{q}$. 

Let $h_{i}(x)=x^{n}+\alpha_{n-1}^{(i)}x^{n-1}+\cdots+\alpha_{0}^{(i)}$
with $\alpha_{v}^{(i)}\in\mathbb{F}_{q}$ for $0\leq v\leq n-1$.
Then for $1\leq d\leq n$, we have $c_{n-d}(0)=a_{0}^{(n-d)}=\alpha_{n-d}^{(i)}.$
It follows that there are $q^{d}$ choices for such $c_{n-d}(y)$
and hence there are $q^{1}q^{2}\cdots q^{n}=q^{n(n+1)/2}$ choices
for such $f(x,y)$. Similarly there are $q^{m(m+1)/2}$ choices for
$g(x,y)$. Let us denote these numbers as $D=q^{n(n+1)/2}$ and $R=q^{m(m+1)/2}$.
Since we have $s$ choices for $i$, $|A_{0}|=sDR$.

On the other hand for a given $\gamma\in\mathbb{F}_{q}$ if $(f(x,y),g(x,y))\in A_{0}$
then $(f(x,y-\gamma),g(x,y-\gamma))\in A_{\gamma},$ since $f(x,y-\gamma)$
is again a bivariate polynomial which is a monic polynomial in $x$
of total degree $n$ and $g(x,y-\gamma)$ is again a bivariate polynomial
which is monic polynomial in $x$ of total degree $m$. This correspondence
(coordinate transformation) is bijective. Hence for any $\gamma\in\mathbb{F}_{q},$
one has $|A_{\gamma}|=sDR.$ 

For a general polynomial $f(x,y)\in\mathbb{F}_{q}[x,y]$ which is
monic in $x$ and of total degree $n>0$, one has $q^{2}q^{3}\cdots q^{n+1}=q^{n}D$
choices. Similarly for a general polynomial $g(x,y)\in\mathbb{F}_{q}[x,y]$
which is monic in $x$ and of total degree $m>0$, one has $q^{2}q^{3}\cdots q^{m+1}=q^{m}R$
choices and therefore there are $q^{n+m}DR$
pairs $(f,g)$ which are monic in $x$ with total degrees $\mathrm{deg}(f)=n$
and $\mathrm{deg}(g)=m$.

%Let $X$ be a random variable that counts the number of $\gamma\in\mathbb{F}_{q}$
%such that $\mathrm{gcd}(f(x,\gamma),g(x,\gamma))\neq1$ for a given
%polynomial pair $(f,g)\in\mathbb{F}_{q}[x,y]^{2}$ where $f,g$ are
%monic in $x$ with total degrees, $\mathrm{deg}(f)=n>0$ and $\mathrm{deg}(g)=m>0$.
Let $x_{i}:=\mathrm{Prob}[X=i]$. This is the probability that
$\mathrm{gcd}(f(x,\gamma),g(x,\gamma))\neq1$ for exactly $i$ different
$\gamma$'s in $\mathbb{F}_{q},$ i.e.  the probability that $(f,g)\in B_{i}$
in the notation introduced in section 1 considering the finite collection
of sets 
$\Gamma=\{A_{\gamma},\gamma \in \FF_q\}$.
Hence $x_{i}=\frac{b_{i}}{q^{n+m}DR}$.
Then by Proposition 1 
\begin{eqnarray*}
\mathrm{E}[X] & = & \sum_{i=0}^{q}ix_{i}=\sum_{i=0}^{q}i\frac{b_{i}}{q^{n+m}DR}=\frac{\sum_{i=0}^{q}ib_{i}}{q^{n+m}DR}\\
 & = & \frac{\sum_{i=0}^{q-1}|A_{i}|}{q^{n+m}DR}=\frac{\sum_{i=0}^{q-1}sDR}{q^{n+m}DR}=\frac{qsDR}{q^{n+m}DR}=\frac{qq^{n+m-1}}{q^{n+m}}=1.
\end{eqnarray*}
To determine the variance of $X$, our proof assumes a set ordering
of the elements of $\FF_q.$  For this purpose let us fix a
generator $\alpha$ of $\FF_q^*$ and use the ordering
$0 < 1 < \alpha < \alpha^2 < \dots < \alpha^{q-2}$.

For $(\gamma,\theta) \in \FF_q^2$ with $\gamma<\theta$,
let us define $A_{\gamma,\theta}$ as the set
of bivariate polynomial pairs $(f,g)$ with $f,g$ are monic in $x$
with total degrees, $\mathrm{deg}(f)=n>0$ and $\mathrm{deg}(g)=m>0$
such that $\mathrm{gcd}(f(x,\gamma),g(x,\gamma))\neq1$ and $\mathrm{gcd}(f(x,\theta),g(x,\theta))\neq1$.
Our first aim is to compute $|A_{0,1}|.$ 

Let $f,g\in A_{0,1}$. Since $f$ and $g$ are monic in $x$, $f(x,0),f(x,1)$
are monic polynomials of degree $n$ and $g(x,0),g(x,1)$ are monic
polynomials of degree $m$ in $\mathbb{F}_{q}[x].$ We have finitely
many choices for non-relatively prime monic polynomial pairs $(h_{i}(x),l_{i}(x))$
with $\mathrm{deg}(h_{i})=n$ and $\mathrm{deg}(l_{i})=m$ with $i=1,\ldots,s$$ $
in $\mathbb{F}_{q}[x]^{2}.$

Let $(f(x,0),g(x,0))=(h_{i}(x),l_{i}(x))$ and $(f(x,1),g(x,1))=(h_{j}(x),l_{j}(x))$
for some fixed pair $(i,j)$ where $1\leq i,j\leq s$. 

Let $f(x,y)=x^{n}+c_{n-1}(y)x^{n-1}+\cdots+c_{1}(y)x+c_{0}(y)$
where $c_{d}(y)\in\mathbb{F}_{q}[y]$ of total degree $\mathrm{deg}(c_{n-d}(y))\leq d$
and let $c_{n-d}(y)=a_{d}^{(n-d)}y^{d}+\cdots+a_{0}^{(n-d)}$
where $a_{i}^{(n-d)}\in\mathbb{F}_{q}$. 

Let $h_{i}(x)=x^{n}+\alpha_{n-1}^{(i)}x^{n-1}+\cdots+\alpha_{0}^{(i)}$
and $h_{j}(x)=x^{n}+\beta_{n-1}^{(j)}x^{n-1}+\cdots+\beta_{0}^{(j)}$
with $\alpha_{v}^{(i)},\beta_{w}^{(j)}\in\mathbb{F}_{q}$ for $0\leq v,w\leq n-1$.
Then for $1\leq d\leq n$, we have 

\smallskip
$c_{n-d}(0)=a_{0}^{(n-d)}=\alpha_{n-d}^{(i)} ~~{\rm and}~~ c_{n-d}(1)=a_{d}^{(n-d)}+\cdots+a_{1}^{(n-d)}+a_{0}^{(n-d)}=\beta_{n-d}^{(j)}.$

\smallskip \noindent
It follows that there are $q^{d-1}$ choices for such $c_{n-d}(y)$
and hence there are $q^{0}q^{1}\cdots q^{n-1}=q^{n(n-1)/2}$ choices
for such $f(x,y)$. Similarly there are $q^{m(m-1)/2}$ choices for
$g(x,y)$. Let us call these numbers as $D_{1}=q^{n(n-1)/2}$ and
$R_{1}=q^{m(m-1)/2}$. Since we have $s^{2}$ choices for $(i,j)$
($i$ and $j$ are need not be different, $|A_{0,1}|=s^{2}D_{1}R_{1}.$ 

On the other hand if $(f(x,y),g(x,y))\in A_{0,1}$ then for
$\gamma,\theta\in\mathbb{F}_{q}$ with $\gamma<\theta$, $(f(x,\frac{y-\gamma}{\theta-\gamma}),g(x,\frac{y-\gamma}{\theta-\gamma}))\in A_{\gamma,\theta}$,
since $f(x,\frac{y-\gamma}{\theta-\gamma})$ is again a monic polynomial
in $x$ of total degree $n$ and $g(x,\frac{y-\gamma}{\theta-\gamma})$
is again a monic polynomial in $x$ of total degree $m$. This correspondence
(coordinate transformation) is bijective and preserves relative primeness.
Hence for a given $\gamma,\theta\in\mathbb{F}_{q}$
with $\gamma<\theta$, one has $|A_{\gamma,\theta}|=s^{2}D_{1}R_{1}.$

For a general bivariate polynomial $f(x,y)\in\mathbb{F}_{q}[x,y]$
which is monic in $x$ and of total degree $n$, one has $q^{2}q^{3}\cdots q^{n+1}=q^{2n}D_{1}$
choices. Similarly for a general bivariate polynomial $g(x,y)\in\mathbb{F}_{q}[x,y]$
which is monic in $x$ and of total degree $m$, one has $q^{2}q^{3}\cdots q^{m+1}=q^{2m}R_{1}$
choices and therefore the number of bivariate polynomial pairs
in $(f,g)$ which are monic in $x$ with total degrees, $\mathrm{deg}(f)=n$
and $\mathrm{deg}(g)=m$ is $q^{2n+2m}D_{1}R_{1}.$ Then with this
notation we have $x_{i}=\frac{b_{i}}{q^{2n+2m}D_{1}R_{1}}.$ 

Since we have $\binom{q}{2}$ choices for $(\gamma,\theta)$ with
$\gamma<\theta$, $|A_{\gamma,\theta}|=s^{2}D_{1}R_{1}$ for all $(\gamma,\theta)$
with $\gamma<\theta$ and $\mathrm{E}[X]=1,$ by Proposition 1 we have

\begin{eqnarray*}
\mathrm{Var}[X] & = & \mathrm{E}[X^{2}]-\mathrm{E}[X]^{2}
 ~ = ~ \mathrm{E}[X^{2}]-1^{2}
 ~ = ~ -1+\sum_{i=0}^{q}i^{2}x_{i}\\
 & = & -1+\sum_{i=0}^{q}i^{2}\frac{b_{i}}{q^{2n+2m}D_{1}R_{1}}
 ~ = ~-1+\frac{\sum_{i=0}^{q}i^{2}b_{i}}{q^{2n+2m}D_{1}R_{1}}\\
 & = & -1+\frac{\sum_{i=0}^{q-1}|A_{i}|+2\sum_{i<j}|A_{i}\cap A_{j}|}{q^{2n+2m}D_{1}R_{1}}
 ~ = ~ -1+\frac{\sum_{i=0}^{q-1}|A_{i}|}{q^{2n+2m}D_{1}R_{1}}+\frac{2\sum_{i<j}s^{2}D_{1}R_{1}}{q^{2n+2m}D_{1}R_{1}}\\
 & = & -1+\frac{\sum_{i=0}^{q}ib_{i}}{q^{2n+2m}D_{1}R_{1}}+\frac{2\sum_{i<j}s^{2}D_{1}R_{1}}{q^{2n+2m}D_{1}R_{1}}
 ~ = ~ -1+\sum_{i=0}^{q}i\frac{b_{i}}{q^{2n+2m}D_{1}R_{1}}+\frac{2\sum_{i<j}s^{2}D_{1}R_{1}}{q^{2n+2m}D_{1}R_{1}}\\
 & = & -1+\sum_{i=0}^{q}ix_{i}+\frac{2\sum_{i<j}s^{2}D_{1}R_{1}}{q^{2n+2m}D_{1}R_{1}}
 ~ = ~ -1+\mathrm{E}[X]+\frac{2\sum_{i<j}s^{2}D_{1}R_{1}}{q^{2n+2m}D_{1}R_{1}}\\
 & = & -1+1+\frac{2\sum_{i<j}s^{2}D_{1}R_{1}}{q^{2n+2m}D_{1}R_{1}}=\frac{2\binom{q}{2}s^{2}D_{1}R_{1}}{q^{2n}q^{2m}D_{1}R_{1}}
 ~ = ~ \frac{q(q-1)q^{2n+2m-2}}{q^{2n+2m}} \\
 & = & \frac{q(q-1)}{q^{2}} 
 ~ = ~1-\frac{1}{q}.  ~~~~~~~~~~~~~~~~~~~~~~~~~~~~~~~~~~~~~~~~~~~ \Box
\end{eqnarray*} 

\begin{thm}
Let $f,g\in\mathbb{F}_{q}[x_{1},x_{2},\ldots,x_{n}]$ be of the form
$f = c_l x_1^l + \sum_{i=0}^{l-1} c_{l-i}(x_2,\dots,x_n) x^i$  and
$g = d_m x_1^m + \sum_{i=0}^{m-1} d_{m-i}(x_2,\dots,x_n) x^i$
where $c_l \ne 0,$ $d_m \ne 0,$ $\deg c_{l-i} \le l-i$, and $\deg d_{m-i} \le m-i$,
thus $f$ and $g$ have total degree $l$ and $m$ respectively.
Let $X$ be a random variable which counts the number 
of $\gamma=(\gamma_{2},\ldots,\gamma_{n})\in\mathbb{F}_{q}^{n-1}$
such that $\mathrm{gcd}(f(x_1,\gamma_2,\ldots,\gamma_n),g(x_1,\gamma_2,\ldots,\gamma_n))\neq1.$ \\
If $n>1$, $l>0$ and $m>0$ then

\begin{itemize}
\item[(a)] $\E[X]=q^{n-2}$ and
\item[(b)] $\Var[X]=q^{n-2}(1-1/q).$
\end{itemize}

\end{thm}
It follows from (a) that if $\gamma$ is chosen at random from $\FF_q^{n-1}$ then
$$
\Prob[ \, \gcd( f(x_1,\gamma_2,\dots,\gamma_n), g(x_2,\gamma_2,\dots,\gamma_n) \ne 1 \, ] \, = \, \frac{q^{n-2}}{q^{n-1}} \, =\,  \frac{1}{q}. 
$$
%\noindent
%{\bf Proof:} \ A version of the paper with the proof which runs about 3 pages may be found at \\
%\verb+http://www.cecm.sfu.ca/~mmonagan/papers/FPSAC16.pdf+
%\begin{thm}
%Let $f,g\in\mathbb{F}_{q}[x_{1},x_{2},\ldots,x_{n}]$ be of the form
%$f = c_l x_1^l + \sum_{i=0}^{l-1} c_{l-i}(x_2,\dots,x_n) x^i$  and
%$g = d_m x_1^m + \sum_{i=0}^{m-1} d_{m-i}(x_2,\dots,x_n) x^i$
%where $c_l \ne 0,$ $d_m \ne 0,$ $\deg c_{l-i} \le l-i$, and $\deg d_{m-i} \le m-i$,
%thus $f$ and $g$ have total degree $l$ and $m$ respectively.
%Let $X$ be a random variable which counts the number 
%of $\gamma=(\gamma_{2},\ldots,\gamma_{n})\in\mathbb{F}_{q}^{n-1}$
%such that $\mathrm{gcd}(f(x_1,\gamma_2,\ldots,\gamma_n),g(x_1,\gamma_2,\ldots,\gamma_n))\neq1.$ \\
%If $n>1$, $l>0$ and $m>0$ then
%
%(a) $\E[X]=q^{n-2}$ and
%
%(b) $\Var[X]=q^{n-2}(1-1/q).$
%%
%\end{cor}
%%
%It follows that if $\gamma$ is chosen at random from $\FF_q^{n-1}$ then
%$$
%\Prob[ \, \gcd( f(x_1,\gamma_2,\dots,\gamma_n), g(x_2,\gamma_2,\dots,\gamma_n) \ne 1 \, ] \, = \, \frac{q^{n-2}}{q^{n-1}} \, =\,  \frac{1}{q}. 
%$$

\begin{proof}
The proof runs along the same lines of the proof of Theorem 2. Let
$U$ be the set of all possible monic pairs $(f,g)\in\mathbb{F}_{q}[x_{1},\ldots,x_{n}]^{2}$
where $f,g$ are as described in the theorem and for $\alpha=(\alpha_{2},\ldots,\alpha_{n})\in\mathbb{F}_{q}^{n-1}$,
let $A_{\alpha}$ be the set of all such polynomial pairs with
$\mathrm{gcd}(f(x_{1},\alpha),g(x_{1},\alpha))\neq1.$ 

For the first part we will consider the monic pairs $(f,g)$ with
\[
\mathrm{gcd}(f(x_{1},0,0\ldots,0),g(x_{1},0,0\ldots,0))\neq1
\]
 and compute that probability of this event is $1/q$ again. Then
for a given non-zero $\alpha=(\alpha_{2},\ldots,\alpha_{n})\in\mathbb{F}_{q}^{n-1}$
considering the coordinate change 
\[
\bar{f}(x_{1},\alpha_{2},\ldots,\alpha_{n})=f(x_{1},x_{2}-\alpha_{2},\ldots,x_{n}-\alpha_{n})
\]
 and using Proposition 1 part (a), since there are $q^{n-1}$ possible
such $\alpha$'s, we will see that $\mathrm{E}[X]=q^{n-1}q^{-1}=q^{n-2}.$

For the second part we will consider the monic pairs $(f,g)$ with
\begin{eqnarray*}
\mathrm{gcd}(f(x_{1},0,0\ldots,0),g(x_{1},0,0\ldots,0)) & \neq & 1\\
\mathrm{gcd}(f(x_{1},1,0\ldots,0),g(x_{1},1,0\ldots,0)) & \neq & 1
\end{eqnarray*}
and see that probability of this event is $1/q^{2}$ again. For a
given pair $(\alpha,\beta)\in\mathbb{F}_{q}^{n-1}\times\mathbb{F}_{q}^{n-1}$
with $\alpha\neq\beta$, this time the coordinate change of the second
part of the proof that computes the variance may not be that obvious.
We give the explicit contruction below. Then by enumarating the elements
of $\mathbb{F}_{q}^{n-1}$ from $0$ to $q^{n-1}-1$ and using Proposition 
1 part (b), since there are $\binom{q^{n-1}}{2}$ possible pairs $(\alpha,\beta)$
with $\alpha<\beta,$ we will see that 
\begin{eqnarray*}
\mathrm{Var}[X] & = & \mathrm{E}[X^{2}]-\mathrm{E}[X]^{2}=-\mathrm{E}[X]^{2}+\sum_{i=0}^{q^{n-1}-1}i^{2}\mathrm{Pr}(X=i)\\
 & = & -\mathrm{E}[X]^{2}+\sum_{i=0}^{q^{n-1}-1}i^{2}\frac{b_{i}}{|U|}=-\mathrm{E}[X]^{2}+\frac{\sum_{i=0}^{q^{n-1}-1}i^{2}b_{i}}{|U|}\\
 & = & -\mathrm{E}[X]^{2}+\frac{\sum_{i=0}^{q^{n-1}-1}|A_{i}|+2\sum_{i<j}|A_{i}\cap A_{j}|}{|U|}\\
 & = & -\mathrm{E}[X]^{2}+\sum_{i=0}^{q^{n-1}-1}\frac{1}{q}+2\sum_{i<j}\frac{1}{q^{2}}\\
 & = & -\mathrm{E}[X]^{2}+q^{n-1}\frac{1}{q}+2\binom{q^{n-1}}{2}\frac{1}{q^{2}}\\
 & = & -q^{2n-4}+q^{n-2}+q^{n-1}(q^{n-1}-1)q^{-2}\\
 & = & q^{n-2}-q^{n-3}=q^{n-2}(1-\frac{1}{q}).
\end{eqnarray*}

For a given pair $(\alpha,\beta)\in\mathbb{F}_{p}^{n-1}\times\mathbb{F}_{p}^{n-1}$
with $\alpha\neq\beta$, let $\alpha=(\alpha_{2},\ldots,\alpha_{n})$
and $\beta=(\beta_{2},\ldots,\beta_{n})$. Our aim is to find a coordinate
change such that
\begin{eqnarray*}
\bar{f}(x_{1},\alpha_{2},\ldots,\alpha_{n}) & = & f(x_{1},0,0\ldots,0)\\
\bar{f}(x_{1},\beta_{2},\ldots,\beta_{n}) & = & f(x_{1},1,0\ldots,0)
\end{eqnarray*}
where

\[
\bar{f}(x_{1},x_{2},\ldots,x_{n})=f(x_{1},a_{20}+a_{22}x_{2}+\cdots+a_{2n}x_{n},\ldots,a_{n0}+a_{n2}x_{n}+\cdots+a_{nn}x_{n}).
\]
Note that this transformation does not change the leading term in
$x_{1}$, so it preserves monicness and degree in $x_{1}$ and preserves
coprimality. To make this transformation bijective we need a $(n-1)\times(n-1)$
matrix
\[
A=\left(\begin{array}{ccc}
a_{22} & \ldots & a_{2n}\\
\vdots & \vdots & \vdots\\
a_{n2} & \ldots & a_{nn}
\end{array}\right)
\]
which is invertible and to satisfy the relations we need 
\begin{eqnarray*}
a_{22}\alpha_{2}+\cdots+a_{2n}\alpha_{n} & = & -a_{20}\\
a_{22}\beta_{2}+\cdots+a_{2n}\beta_{n} & = & 1-a_{20}
\end{eqnarray*}
 and for $3\leq j\leq n$ 
\begin{eqnarray*}
a_{j2}\alpha_{2}+\cdots+a_{jn}\alpha_{n} & = & -a_{j0}\\
a_{j2}\beta_{2}+\cdots+a_{jn}\beta_{n} & = & -a_{j0}.
\end{eqnarray*}
Let us consider $\alpha$ and $\beta$ as column vectors and suppose
that $\alpha$ and $\beta$ are linearly independent over $\mathbb{F}_{p}$.
Then there exist a pair $(i,j)$ such that $\left|\begin{array}{cc}
\alpha_{i} & \alpha_{j}\\
\beta_{i} & \beta_{j}
\end{array}\right|\neq0.$ Applying the necessary permutation if needed, we may assume that
$\left|\begin{array}{cc}
\alpha_{2} & \alpha_{3}\\
\beta_{2} & \beta_{3}
\end{array}\right|\neq0.$ Then consider the invertible $(n-1)\times(n-1)$ matrix $B$
\[
B:=\left(\begin{array}{ccccc}
\alpha_{2} & \alpha_{3} & \ldots & \ldots & \alpha_{n}\\
\beta_{2} & \beta_{3} & \ldots & \ldots & \beta_{n}\\
0 & 0 & 1 & \ldots & 0\\
\vdots & \vdots & \vdots & \vdots & \vdots\\
0 & 0 & 0 & \ldots & 1
\end{array}\right)
\]
Let $a_{j}$ denotes $a_{j}=(a_{22}\,\cdots\,a_{2n})^{T}$ for $2\leq j\leq n$.
Let $a_{2}:=B^{-1}\left(-1\:0\,\cdots\,0\right)^{T}$, so that $a_{2}\cdot\alpha^{T}=-1,\,a_{2}\cdot\beta^{T}=0$.
Let $a_{3}:=B^{-1}\left(1\,\mathbf{e}_{1}^{T}\right)^{T}$ so that
$a_{3}\cdot\alpha^{T}=1,\,a_{3}\cdot\beta^{T}=1$ and let $a_{j}:=B^{-1}(0\,\mathbf{e}_{j-2}^{T})^{T}$
so that $a_{j}\cdot\alpha^{T}=0,\,a_{j}\cdot\beta^{T}=0$ for $3<j<n$
where $\mathbf{e}_{i}$'s denote canonical basis vectors for $\mathbb{Z}_{p}^{n-2}.$
Let also $a_{20}=1,a_{30}=-1$ and $a_{j0}=0.$ Now, if we define
$A=\left(a_{2}\,\cdots\,a_{n}\right)^{T}$ then by construction of
$a_{i}$'s we have\\
 $A\cdot\alpha=\left(\begin{array}{c}
a_{2}\cdot\alpha^{T}\\
a_{3}\cdot\alpha^{T}\\
\vdots\\
a_{n}\cdot\alpha^{T}
\end{array}\right)=\left(\begin{array}{c}
-1\\
1\\
0\\
\vdots\\
0
\end{array}\right)$ and $A\cdot\beta=\left(\begin{array}{c}
a_{2}\cdot\beta^{T}\\
a_{3}\cdot\beta^{T}\\
\vdots\\
a_{n}\cdot\beta^{T}
\end{array}\right)=\left(\begin{array}{c}
0\\
1\\
0\\
\vdots\\
0
\end{array}\right)$. Hence we get $a_{2}\cdot\alpha^{T}+a_{20}=-1+1=0$ and $a_{2}\cdot\beta^{T}+a_{20}=0+1=1$
as needed. Also $a_{3}\cdot\alpha^{T}+a_{30}=1-1=0$ and $a_{3}\cdot\beta^{T}+a_{30}=1-1=0$
as needed. Also $a_{j}\cdot\alpha^{T}+a_{j0}=0+0=0$ and $a_{j}\cdot\beta^{T}+a_{j0}=0+0=0$
as needed. It remains to show that the set $\left\{ a_{2},a_{3},\ldots,a_{n}\right\} $
is linearly independent.

Now since $B$ is invertible the set $\left\{ a_{2},a_{3},\ldots,a_{n}\right\} $
is linearly independent iff the set $\left\{ Ba_{2},Ba_{3},\ldots,Ba_{n}\right\} $
is linearly independent. Let for some $\gamma_{i}\in\mathbb{F}_{p},2\leq i\leq n$
we have 
\[
\gamma_{2}\left(\begin{array}{c}
-1\\
0\\
0\\
\vdots\\
0
\end{array}\right)+\gamma_{3}\left(\begin{array}{c}
1\\
1\\
0\\
\vdots\\
0
\end{array}\right)+\gamma_{4}\left(\begin{array}{c}
0\\
0\\
1\\
\vdots\\
0
\end{array}\right)+\cdots+\gamma_{n}\left(\begin{array}{c}
0\\
0\\
0\\
\vdots\\
1
\end{array}\right)=0.
\]
Then it can be easily seen that $\gamma_{i}=0$ and hence $\left\{ a_{2},a_{3},\ldots,a_{n}\right\} $
is a linearly independent set. It follows that $A$ is invertible
and the translation we have constructed 
\[
\left(\begin{array}{ccc}
a_{22} & \ldots & a_{2n}\\
\vdots & \vdots & \vdots\\
a_{n2} & \ldots & a_{nn}
\end{array}\right)\left(\begin{array}{c}
x_{2}\\
\vdots\\
x_{n}
\end{array}\right)+\left(\begin{array}{c}
1\\
-1\\
0\\
\vdots\\
0
\end{array}\right)
\]
satisfies the conditions we needed. \\
Example: Let $\alpha=(2,0,0,0)$ and $\beta=(2,3,0,1)$ in $\mathbb{Z}_{5}^{4}.$
Then 
\[
B=\left(\begin{array}{cccc}
2 & 0 & 0 & 0\\
2 & 3 & 0 & 1\\
0 & 0 & 1 & 0\\
0 & 0 & 0 & 1
\end{array}\right)\Rightarrow B^{-1}=\left(\begin{array}{cccc}
3 & 0 & 0 & 0\\
3 & 2 & 0 & 3\\
0 & 0 & 1 & 0\\
0 & 0 & 0 & 1
\end{array}\right)\,\mathrm{mod}\,5.
\]
$a_{2}=B^{-1}\left(\begin{array}{c}
-1\\
0\\
0\\
0
\end{array}\right)=\left(\begin{array}{c}
2\\
2\\
0\\
0
\end{array}\right),a_{3}=B^{-1}\left(\begin{array}{c}
1\\
1\\
0\\
0
\end{array}\right)=\left(\begin{array}{c}
3\\
0\\
0\\
0
\end{array}\right),a_{4}=B^{-1}\left(\begin{array}{c}
0\\
0\\
1\\
0
\end{array}\right)=\left(\begin{array}{c}
0\\
0\\
1\\
0
\end{array}\right),a_{5}=B^{-1}\left(\begin{array}{c}
0\\
0\\
0\\
1
\end{array}\right)=\left(\begin{array}{c}
0\\
3\\
0\\
1
\end{array}\right).$ Then the transformation is
\[
\left(\begin{array}{cccc}
2 & 2 & 0 & 0\\
3 & 0 & 0 & 0\\
0 & 0 & 1 & 0\\
0 & 3 & 0 & 1
\end{array}\right)\left(\begin{array}{c}
x_{2}\\
x_{3}\\
x_{4}\\
x_{5}
\end{array}\right)+\left(\begin{array}{c}
1\\
-1\\
0\\
0
\end{array}\right)
\]
 and $\bar{f}(x_{1},x_{2},x_{3},x_{4},x_{5})=f(x_{1},1+2x_{2}+2x_{3},-1+3x_{2},x_{4},3x_{3}+x_{5})$
and 
\begin{eqnarray*}
\bar{f}(x_{1},2,0,0,0) & = & f(x_{1},0,0,0,0)\\
\bar{f}(x_{1},2,3,0,1) & = & f(x_{1},1,0,0,0).
\end{eqnarray*}
Now suppose that $\alpha$ and $\beta$ are linearly dependent over
$\mathbb{F}_{p}$. Again applying the necessary permutation if needed,
we may assume that $0\neq\alpha_{2}\,\mathrm{and}\,\alpha_{2}\neq\beta_{2}.$
Let $a_{20}=-\alpha_{2}/(\beta_{2}-\alpha_{2})$ ,$a_{22}=1/(\beta_{2}-\alpha_{2})$
and $a_{2j}=0$ for $3\leq j\leq n.$ Then we have $a_{2}\cdot\alpha^{T}+a_{20}=\frac{\alpha_{2}}{\beta_{2}-\alpha_{2}}-\frac{\alpha_{2}}{\beta_{2}-\alpha_{2}}=0$
and $a_{2}\cdot\beta^{T}+a_{20}=\frac{\beta_{2}}{\beta_{2}-\alpha_{2}}-\frac{\alpha_{2}}{\beta_{2}-\alpha_{2}}=1$
as needed. \\
Now consider the $1\times(n-1)$ matrix $B=\left(\alpha_{2}\cdots\alpha_{n}\right).$
Since $\alpha\neq0$ $\mathrm{dim}(\mathrm{Ker}(B))=n-2.$ Let $v_{2},\ldots,v_{n}$
be a basis for $\mathrm{Ker}(B).$ Let for some $\gamma_{i}\in\mathbb{F}_{p},2\leq i\leq n$
we have 
\[
\gamma_{2}a_{2}+\gamma_{3}v_{3}+\gamma_{4}v_{4}+\cdots+\gamma_{n}v_{n}=0.
\]
 Then applying $B$ from left hand side we have $\gamma_{2}Ba_{2}=\gamma_{2}\alpha_{2}/(\beta_{2}-\alpha_{2})=0\Rightarrow\gamma_{2}=0.$
It follows that $\gamma_{i}=0$ and $\{a_{2},v_{2},\ldots,v_{n}\}$
is linearly independent. Then $A:=\left(a_{2}\,v_{2}\,\cdots\,v_{n}\right)^{T}$
is invertible. It can be readily verified that the translation we
have constructed
\[
A\left(\begin{array}{c}
x_{2}\\
\vdots\\
x_{n}
\end{array}\right)+\left(\begin{array}{c}
\frac{-\alpha_{2}}{\beta_{2}-\alpha_{2}}\\
0\\
\vdots\\
0
\end{array}\right)
\]
satisfies the remaining conditions needed. \\
Example: Let $\alpha=(1,1,0,0)$ and $\beta=(2,2,0,0)$ in $\mathbb{Z}_{5}^{4}.$
Then $a_{20}=-\alpha_{2}/(\beta_{2}-\alpha_{2})=4$ and $a_{22}=1/(\beta_{2}-\alpha_{2})=1\:\mathrm{mod}\:5$.
$B=\left(1\:1\,0\,0\right)$
\[
v_{3}=\left(\begin{array}{c}
3\\
2\\
0\\
0
\end{array}\right),v_{4}=\left(\begin{array}{c}
0\\
0\\
1\\
0
\end{array}\right),v_{5}=\left(\begin{array}{c}
0\\
0\\
0\\
1
\end{array}\right)
\]
 and then the transformation is
\[
\left(\begin{array}{cccc}
1 & 0 & 0 & 0\\
3 & 2 & 0 & 0\\
0 & 0 & 1 & 0\\
0 & 0 & 0 & 1
\end{array}\right)\left(\begin{array}{c}
x_{2}\\
x_{3}\\
x_{4}\\
x_{5}
\end{array}\right)+\left(\begin{array}{c}
4\\
0\\
0\\
0
\end{array}\right)
\]
 where $\bar{f}(x_{1},x_{2},x_{3},x_{4},x_{5})=f(x_{1},4+x_{2},3x_{2}+2x_{3},x_{4},x_{5})$
and 
\begin{eqnarray*}
\bar{f}(x_{1},1,1,0,0) & = & f(x_{1},0,0,0,0)\\
\bar{f}(x_{1},2,2,0,0) & = & f(x_{1},1,0,0,0).
\end{eqnarray*}

\end{proof}

\subsection{A comparison with the binomial distribution.}

Let $Y$ be a random variable from a binomial distribution $B(n,p)$
with $n$ trials and probability $p$.  So $0 \le Y \le n$,
$\Prob[Y=k] = {n \choose k} p^k (1-p)^{n-k}$,  $E[Y]=np$ and $\Var[Y]=np(1-p)$.
We noticed that the mean and variance of $X$ in Theorem 2 is the same as the mean and variance of 
the binomial distribution $B(n,p)$ with $n=q$ trials and probability $p=1/q.$
In Table 1 below we compare the two distributions for

\medskip \quad
\begin{tabular}{ll}
$ f = x^2 + (a_1 y + a_2 )x + (a_3 y^2 + a_4 y + a_5)$ and \\
$ g = x^2 + (b_1 y + b_2 )x + (b_3 y^2 + b_4 y + b_5)$
\end{tabular}

\medskip \noindent
in $\FF_q[x,y]$ with $q=7$.
Note that there are $7^{10}$ pairs for $f,g$.
In Table 1 $F_k$ is the number of pairs for which $\gcd(f(x,\alpha),g(x,\alpha)) \ne 1$
for exactly $k$ values for $\alpha \in \FF_7$.
We computed $F_k$ by computing this $\gcd$ for all distinct pairs using Maple.
The values for $B_k$ come from $B(7,1/7)$.  They are given by $B_k = 7^{10} \Prob[Y=k]$.

\newpage

% MBM: F_k data in file cr22.mag.out and cr.mpl.out
% "frequencies", 96606636, 110666892, 56053746, 17287200, 1728720, 0, 0, 132055
\begin{table}[h!]
{ \small
\begin{tabular}{|r r r r r r r r r|} \hline
$k  $ &    0   &    1  &      2  &      3 &       4 &     5&  6  & 7    \\
$F_k$ & 96606636& 110666892& 56053746& 17287200& 1728720& 0&  0  & 132055 \\
$B_k$ & 96018048& 112021056& 56010528& 15558480& 2593080& 259308& 14406& 343 \\  \hline
\end{tabular} \vspace*{-1mm}
\caption{Data for quadratic $(f,g)$ in $\FF_7[x,y]$}
}
\end{table}

\vspace*{-1mm}
\noindent
The two zeros $F_{5}$ and $F_{6}$ can be explained as follows.
Let $R(y)$ be the Sylvester resultant of $f$ and $g$.
Then applying Lemma 1 we have
$ R(\alpha) = 0 \iff \gcd(f(x,\alpha),g(x,\alpha)) \ne 1 ~{\rm for}~ \alpha \in \FF_q.$
For our quadratic polynomials $f$ and $g$, Lemma 1(ii) implies
$\deg R \le \deg f \deg g = 4$.
%
%\[
% \left[ \begin {array}{cccc} 1&a_1y+a_2&a_3{y}^2+a_4y+a_5&0\\
% \noalign{\medskip}0&1&a_1y+a_2& a_3y^2+a_4y+a_5\\
% \noalign{\medskip}1&b_1y+b_2&b_3{y}^2+b_4y+b_5&0\\
%  \noalign{\medskip}0&1&b_1y+b_2&b_3{y}^2+b_4y+b_5\end {array} \right]
%\]
%
%\[
%\left[ \begin {array}{cccc} 1&0&1&0 \\
%  \noalign{\medskip}a_1y+a_2&1&b_1y+b_2&1 \\
%  \noalign{\medskip}a_3y^2+a_4y+a_5&a_1y+a_2&b_3y^2+b_4y+b_5& b_1y+b_2\\
%  \noalign{\medskip}0&a_3y^2+a_4y+a_5&0&b_3y^2+b_4y+b_5 
%       \end {array} \right] 
%\]
%
%
%
%\noindent
%By inspection we have $\deg R \le 4$ 
Hence $R(y)$ can have at most 4 distinct roots unless $f$ and $g$ are not coprime in $\FF_7[x,y]$ in
which case $R(y)=0$ and it has 7 roots.  Therefore $F_5=0,$ $F_6=0$ and $F_{7} = 132055$ 
is the number pairs $f,g$ which are not coprime in $\FF_7[x,y]$.

% MBM: I checked this number 132055 explicitly -- see gcd.mag.out

\bibliographystyle{plain}

\section*{Appendix A}

Below is Magma code for quadratic polynomials over $\FF_4$.
For each pair of quadratic polynomials $F,G \in \FF_4[x,y]$ we compute
$X = | \{ \alpha \in \FF_4 : \gcd(F(x,\alpha),G(x,\alpha) \ne 1 \} |.$
The code counts $A_k$ the number of pairs $(F,G)$ with $k=X$ and computes
${\rm E}[X]$ and ${\rm Var}[X]$.

{\small
\begin{verbatim}
q := 4;
Fq<z> := FiniteField(q);
P<x,y> := PolynomialRing(Fq,2);
N := 0; // counter
M := 0; // mean
V := 0; // variance
A := AssociativeArray(); // frequencies

for X in [0..q] do A[X] := 0; end for;

for a in Fq do for b in Fq do
for c in Fq do for d in Fq do for e in Fq do
for r in Fq do for s in Fq do 
for t in Fq do for u in Fq do for v in Fq do

if not ( [a,b,c,d,e] gt [r,s,t,u,v] ) then
    X := 0;
    for y in Fq do
        F := x^2+(a*y+b)*x+(c*y^2+d*y+e);
        G := x^2+(r*y+s)*x+(t*y^2+u*y+v);
        if Gcd(F,G) ne 1 then X := X+1; end if;
    end for;
    if [a,b,c,d,e] eq [r,s,t,u,v] then
        N := N+1; A[X] := A[X]+1;
        M := M+X; V := V+(X-1)*(X-1);
    else
        N := N+2;   A[X] := A[X]+2;
        M := M+2*X; V := V+2*(X-1)*(X-1);
    end if;
end if;

end for; end for; end for; end for; end for;
end for; end for; end for; end for; end for;

"field size", q;
"N", N, q^10;
"frequencies", A[0],A[1],A[2],A[3],A[4];
"mean", 1.0*M/q^10;
"variance", 1.0*V/q^10;
\end{verbatim}
}

\end{document}